\documentclass[a4paper,11pt]{article}
\usepackage[utf8]{inputenc}
\usepackage{amsmath,amssymb,amsthm,enumerate,amsthm,flafter}

\DeclareMathOperator{\Lp}{L}
\DeclareMathOperator{\Ehr}{Ehr}
\DeclareMathOperator{\conv}{conv}

\theoremstyle{plain}
\newtheorem{theorem}{Theorem}[section]
\newtheorem{proposition}[theorem]{Proposition}
\newtheorem{corollary}[theorem]{Corollary}
\newtheorem{conjecture}[theorem]{Conjecture}
\newtheorem{lemma}[theorem]{Lemma}

\theoremstyle{definition}

\newcommand{\eraselater}[1] 
{}
  
\title{Stapledon Decompositions and Inequalities for Coefficients of Chromatic Polynomials}
\author{Emerson León\\
Universidad de los Andes\\ 
Bogotá, Colombia\\
\texttt{ej.leon@uniandes.edu.co}}

\begin{document}

\maketitle
\begin{abstract}
We use a polynomial decomposition result by Stapledon to show that the numerator polynomial of the Ehrhart series of an open polytope is the difference of two symmetric polynomials with nonnegative integer coefficients. We obtain a related decomposition for order polytopes and for the numerator polynomial of the corresponding series for chromatic polynomials. The nonnegativity of the coefficients in such decompositions provide inequalities satisfied by the coefficients of chromatic polynomials for any simple graph. 
\end{abstract}

\section{Introduction}
Let $P$ be a $d$-dimensional polytope with vertices in $\ensuremath{\mathbb{Z}}^d$, with Ehrhart polynomial  $\Lp_P(n)$ 
counting the number of integer points in $nP$
(\cite{ehrhart}) and Ehrhart series  $\Ehr_P(z)=1+\sum_{n=1}^{\infty}\Lp_P(n)z^n$. It is known that the Ehrhart series is a rational function of the form 
$$\Ehr_P(z)=\frac{h^*(z)}{(1-z)^{d+1}}$$ 
where $h^*(z)$ is a polynomial with nonnegative integer coefficients, of degree $s=\deg(h^*)\le d$ (see, for example,  \cite{beckrobins}). In 2008, Stapledon \cite{stapledon1} proved that if $l=d+1-s$ then one can find two symmetric polynomials $a(z)=z^da\left(\frac{1}{z}\right)$ and $b(z)=z^{d-l}b\left(\frac{1}{z}\right)$ with nonnegative integer coefficients such that 
$$(1 + z + \cdots + z^{l-1})h^*(z) = a(z) + z^lb(z).$$ 

In general, any polynomial can be written as the sum of two symmetric polynomials \cite[Lemma 2.3]{stapledon1}, but having nonnegative coefficients $a_i\ge 0$ and $b_i\ge 0$ is what is special in the case of $h^*$-polynomials of convex lattice polytopes. If $l=1$, this is stronger than the fact that polynomials $h^*$ have nonnegative coefficients. Some of these inequalities were already known due to results in commutative algebra  \cite{hibi},  \cite{StanleyHFCMdomain}, and the work by Stapledon provided a new geometric proof of these facts.

Denote by $P^{\circ}$ the interior of $P$ and by  $\Lp_{P^{\circ}}(n)$ the number of lattice points in the interior of $nP$. 
Similarly, we denote $\Ehr_P^{\circ}(z)=\sum_{n= 1}^{\infty}\Lp_{P^{\circ}}(n)z^n$.
Our first result is a polynomial decomposition for the numerator polynomial of  $\Ehr_P^{\circ}(z)$ for any lattice polytope.

\begin{theorem}\label{th3}%
Let $P$ be a lattice polytope of dimension $d$.  Let $h_{P}$ be the numerator polynomial of the Ehrhart series of the open polytope $P^{\circ}$.
Then $h_{P}$ can be decomposed as $$h_{P}(z)=a_P(z)-b_P(z)$$ where $a_P(z)=z^{d+1}a_{P}\left(\frac{1}{z}\right)$ and $b_P(z)=z^{d}b_{P}\left(\frac{1}{z}\right)$ are symmetric polynomials with nonnegative coefficients.
\end{theorem}

An interesting question to ask is if it is possible to get similar decompositions for other combinatorial polynomials. There is a close relationship between the chromatic polynomial of a graph $G$ and Ehrhart polynomials 
of order polytopes corresponding to different posets obtained from $G$ by fixing an acyclic orientation. 


In case that $P$ is an order polytope we can modify Theorem \ref{th3} and prove the following result.
\begin{theorem}\label{th1} 
Let $O_{\Pi}$ be an order polytope of dimension $d$ (corresponding to a poset $\Pi$ on $d$ vertices).  Let $h_{\Pi}$ be the numerator polynomial of the Ehrhart series of $O^{\circ}_{\Pi}$, i.\ e. such that $$\Ehr_{O^{\circ}_{\Pi}}(z)=\frac{h_{\Pi}(z)}{(1-z)^{d+1}}.$$ Then $h_{\Pi}$ can be decomposed as $h_{\Pi}(z)=a_\Pi(z)+zb_\Pi(z)$ where the polynomials $a_\Pi(z)$ and $b_\Pi(z)$ are symmetric with $a_\Pi(z)=z^{d+1}a_{\Pi}\left(\frac{1}{z}\right)$ and $b_\Pi(z)=z^{d}b_{\Pi}\left(\frac{1}{z}\right)$, and so that $b_{\Pi}(z)$ and $-a_{\Pi}(z)$ have nonnegative coefficients.
\end{theorem}

Let $G$ be a graph with $d$ vertices and chromatic polynomial $\chi_G(n)$. Define   $h_G(z)$ to be the polynomial such that $$\sum_{n= 0}^{\infty}\chi_G(n)z^n=\frac{ h_G(z)}{(1-z)^{d+1}}.$$ 

\begin{theorem}\label{th2}
Let $G$ be a graph on $d$ vertices. Then the polynomial $zh_G(z)$ can be decomposed as $zh_G(z)=a(z)+zb(z)$ where $a(z)=z^{d+1}a\left(\frac{1}{z}\right)$ and $b(z)=z^{d}b\left(\frac{1}{z}\right)$ are symmetric, and so that $b(z)$ and $-a(z)$ have nonnegative coefficients.
\end{theorem}
The polynomial $zh_G(z)$ appears naturally as the numerator polynomial for the sum of the ehrhart series for a set of open order polytopes related with the graphs. The fact that the coefficients of $a(z)$ and $b(z)$ are nonnegative can be used to obtain the following inequalities.

\begin{theorem}\label{cor}
 For any graph $G$ with $d$ vertices, $h_G(z)$ is a polynomial of degree $d$ with nonnegative coefficients. Also, if $h_G(z)=\sum_{i=0}^{d} h_iz^i$, then for $1\le i\le \lfloor \frac{d+1}{2}\rfloor$, 
 $$h_{d}+\cdots+h_{d-i+1}-h_0-\cdots-h_{i-1}\ge 0.$$
\end{theorem}
Since it is possible to write the coefficients of $h_G$ as linear combinations in the coefficients of $\chi_G$, Corollary \ref{cor} gives us a family of inequalities that must be satisfied  by the coefficients of chromatic polynomials.  


We continue as follows. In Section 2 we go through the definition and main properties of order polytopes and their relationship with chromatic polynomials. In Section 3 we state some basic facts of Ehrhart theory and the Stapledon decomposition from \cite{stapledon1}.  Section 4 contains the proofs of theorems \ref{th3} and \ref{th1} while in Section 5 we prove Theorem \ref{th2} and Corollary \ref{cor}. Section 6 include some further conjectures that might improve the inequalities from Theorem \ref{cor} (See Conjecture \ref{conjcor}). 

\section{Order Polytopes and Chromatic Polynomials}

Here we recall some basic facts about order polytopes and their relationship with chromatic polynomials and Ehrhart Theory. 
Let $\Pi$ be a poset on $d$ elements $p_1, \ldots, p_d$. 
  Denote by  $\Omega_\Pi(n)$ the number of \emph{weak order-preserving} maps from $\Pi$ to  $\{1,\ldots, n\}$. These are maps  $\phi:\Pi\rightarrow \{1,\ldots, n\}$ so that if $p_i<_\Pi p_j$, then $\phi(p_i)\le \phi(p_j).$
  Also let $\Omega^{\circ}_{\Pi}(n)$ be the number of \emph{strict order-preserving} maps $\Pi$ to  $\{1,\ldots, n\}$ (so that if $p_i<_\Pi p_j$, then $\phi(p_i)<  \phi(p_j)$). These functions are known as \emph{order polynomials}, due to the following result.
 
\begin{theorem}[Stanley \cite{stanleyCLP}]
The functions $\Omega_\Pi(n)$ and $\Omega^{\circ}_{\Pi}(n)$ are polynomials in $n$, such that $$\Omega^{\circ}_{\Pi}(n)=(-1)^d \Omega_\Pi(-n).$$
\end{theorem}

 The \emph{order polytope} $O_{\Pi}$ corresponding to a poset $\Pi$ is defined as  
  $$O_{\Pi}=\{x\in \ensuremath{\mathbb{R}}^d: 0\le x_i \le 1 \text{ for all $i$, and }x_i \le x_j \text{ whenever }p_i \le_{\Pi} p_j \}.$$
\begin{theorem}[Stanley, \cite{stanleyOP}]\label{stanley}
The polynomials $\Omega_\Pi(n)$ and $\Omega^{\circ}_{\Pi}(n)$  are related to the Ehrhart polynomial of the order polytope $O_\Pi$ by the formulas $$\Lp_{O_\Pi}(n)=\Omega_\Pi(n+1),$$ $$\Lp_{O^{\circ}_\Pi}(n)=\Omega^{\circ}_\Pi(n-1).$$ 
\end{theorem}

 Let $G=(V,E)$ be a graph with vertices $V=\{v_1, \ldots, v_d\}$. The \emph{chromatic polynomial} $\chi_G(n)$ counts the maps $\phi:V\rightarrow \{1,\ldots, n\}$ such that $\phi(v_i)\neq \phi(v_j)$ whenever $v_iv_j$ is an edge of $G$. An \emph{orientation} of $G$ is given by a subset $\rho\subseteq E$ of the edges of $G$, so that an edge $v_iv_j\in E$ with $i<j$ is oriented from $v_i$ to $v_j$ if $\{v_i, v_j\}\notin \rho$ or from $v_j$ to $v_i$ if $\{v_i, v_j\}\in \rho$.
 
 An \emph{oriented path} is a sequence of vertices $v_{a_1}, \ldots, v_{a_k}$ so that there is an oriented edge from $v_{a_i}$ to $v_{a_{i+1}}$. An oriented cycle is an oriented path so that $v_{a_1}=v_{a_k}$. An \emph{acyclic orientation} of a graph $G$ is an orientation of $G$ without oriented cycles. Each acyclic orientation $\rho$ of $G$ defines a poset $\Pi_{\rho}$ on the vertex set $V$, where $v_i\le_{\rho} v_j$ if there is an oriented path $v_{a_1}, \ldots, v_{a_k}$ from $v_i=v_{a_1}$ to $v_j=v_{a_k}$. Since $\rho$ is acyclic, the relationship $\le_\rho$ is transitive and antisymmetric. It is also reflexive if we admit oriented paths of length zero (with $k=1$).

The following is a description of the chromatic polynomials in terms of acyclic orientations.
\begin{proposition}[Stanley, \cite{stanleyAO}]\label{chromatic as ao} The chromatic polynomial $\chi_G(n)$ of a graph $G$ is the sum of the strict order polynomials for all acyclic orientations $\rho$ of $G$
$$\chi_G(n)=\sum_{\rho\text{ acyclic}}\Omega^{\circ}_{\Pi_{\rho}}(n).$$
\end{proposition}

\section{Ehrhart $h^*$-vectors and Stapledon Decomposition}

A \emph{triangulation} of a polytope $P$ is a collection of simplices so that their union is $P$ and the intersection of two simplices is a face on the boundary of both simplices (maybe empty). A triangulation of $P$ is \emph{unimodular} if all simplices have integer vertices and minimal volume $1/d!$.
Any order polytope has a unimodular triangulation  by subdividing it through all hyperplanes of the form $x_i=x_j.$ Each of these simplices is an order polytope by itself, corresponding to a poset that is a total order on the $d$ elements.

Let $P$ be a lattice $d$-polytope with Ehrhart series $\Ehr_P(z)=\dfrac{h^*(z)}{(1-z)^{d+1}}.$ 
If $P$ has a unimodular triangulation $T$, it is known that the $h^*$ polynomial can be obtained from the $f$-vector of $T$. If $f_i$ counts the number of $i$-dimensional faces of $T$ (and set $f_{-1}=1$ for the empty face), we construct the $f$-polynomial 
$$f_T(z)=\sum_{i=0}^{d+1} f_{i-1}z^i.$$ 
We define the $h$-polynomial of $T$ to be 
\begin{equation}\label{h}
h_T(z)=(1-z)^{d+1}f_T\left(\frac{z}{1-z}\right).
\end{equation}

Then $h^*(z)=h_T(z)$ is precisely the $h$-polynomial for any unimodular triangulation $T$ of $P$ (see Betke--McMullen \cite{betkemcmullen}).

The following is known as the Ehrhart--Macdonald Reciprocity for Ehrhart series. A proof can be found, e.\ g., in \cite[Theorem 4.4]{beckrobins}. 
\begin{theorem}[\cite{macdonald}]\label{EMreciprocity}
If $P$ is a lattice $d$-polytope, then $\Ehr_{P^{\circ}}(z)=(-1)^{d+1}\Ehr_P\left(\frac{1}{z}\right).$
\end{theorem}

\begin{corollary}\label{reversing}
If $P$ is a lattice $d$-polytope and $\Ehr_{P^{\circ}}(z)=\dfrac{h_P(z)}{(1-z)^{d+1}}$, then $h_P(z)=z^{d+1}h^*\left(\frac{1}{z}\right)$ is a monic polynomial of degree $d+1$.
\end{corollary}
\begin{proof}
By Theorem \ref{EMreciprocity},  $\Ehr_{P^{\circ}}(z)=\dfrac{h^*\left(\frac{1}{z}\right)}{\left(\frac 1z- 1\right)^{d+1}}=\dfrac{z^{d+1}h^*\left(\frac{1}{z}\right)}{(1-z)^{d+1}}.$  
Then $h_P(z)=z^{d+1}h^*\left(\frac{1}{z}\right)$ is obtained by reversing the coefficients of the $h^*$-polynomial. Since  $h^*_0=1$ for any polytope $P$, then $h_P(z)$ is monic and $\deg(h_P)=d+1$.
 \end{proof}

From \cite[Lemma 2.3]{stapledon1} we know the following.
\begin{proposition}\label{abdecomp} Let $h(z)$ be any polynomial of  degree $\deg(h)=s\le d$, and $l=d+1-s$. Then there are unique symmetric polynomials $a(z)$ and $b(z)$ with $a(z)=z^{d}a\left(\frac{1}{z}\right)$ and $b(z)=z^{s-1}b\left(\frac{1}{z}\right)$ so that $h(z) = a(z) + z^lb(z).$  
\end{proposition}

Moreover,  \cite{stapledon1} gives explicit formulas for the coefficients of the polynomials $a(z)=\sum_{i=0}^d a_iz^i$ and $b(z)=\sum_{i=0}^{s-1} b_iz^i$. If $h(z)=\sum_{i=0}^d h_iz^i$, then 
\begin{align}
a_i &=h_0+h_1+\cdots+h_i-h_d-\cdots-h_{d-i+1},\label{abeq1}\\
b_i &=-h_0-h_1-\cdots-h_i+h_s+\cdots+h_{s-i}.\label{abeq2}
\end{align}

It is easy to see that the polynomials $a(z)$ and $b(z)$ obtained this way are symmetric, with $a_i=a_{d-i}$ and $b_i=b_{s-1-i}$. They have integer coefficients if $h(z)$ has integer coefficients.

\begin{theorem}[Stapledon \cite{stapledon1}]\label{stapledon}
Let $h^*(z)$ be the numerator polynomial of the Ehrhart series of a $d$-dimensional lattice polytope $P$. Let $s=\deg(h^*)$ and $l=d+1-s$. Consider the symmetric polynomials $a(z)=z^da\left(\frac{1}{z}\right)$ and $b(z)=z^{s-1}b\left(\frac{1}{z}\right)$  such that $$(1 + z + \cdots + z^{l-1})h^*(z) = a(z) + z^lb(z)$$ 
which exist and are unique due to Proposition \ref{abdecomp}. Then  $a(z)$ and $b(z)$ have nonnegative integer coefficients. 
\end{theorem}

\section{Polynomial Decompositions for Open Lattice Polytopes}

\begin{proof}[Proof of Theorem \ref{th3}]
Consider the Stapledon decomposition of the polytope $P$ given by symmetric polynomials 
$a^*(z)=z^da^*\left(\frac{1}{z}\right)$, $b^*(z)=z^{s-1}b^*\left(\frac{1}{z}\right)$ so that
\begin{equation}\label{eq1b}
(1+\cdots+z^{l-1})h^*(z)=a^*(z)+z^l b^*(z). 
\end{equation}
where $h^*(z)$ is the numerator polynomial of $\Ehr_P(z)$. By Theorem \ref{stapledon}, $a^*(z)$ and $b^*(z)$ have nonnegative coefficients. 

We also apply Theorem \ref{stapledon} to a polytope $\widehat P$ that is a pyramid over $P$ of height one. More precisely, take $\widehat P=\conv(\{0\}\cup \{(v,1):v \text{ vertex of }P \}).$ It is known that $P$ and $\widehat P$ have the same $h^*$ polynomial. (See for example \cite[Theorem 2.4]{beckrobins}.) 

Then we can find symmetric polynomials $\widehat a(z)=z^{d+1}\widehat a\left(\frac{1}{z}\right)$, $\widehat b(z)=z^{s-1}\widehat b\left(\frac{1}{z}\right)$ with nonnegative coefficients, so that
\begin{equation}\label{eq2b}
(1+\cdots+z^{l})h^*(z)=\widehat a(z)+z^{l+1} \widehat b(z). 
\end{equation}
Notice that  $b^*(z)=\widehat b(z),$ since the formula (\ref{abeq2}) does not depend on $d$, but only on $s$. Multiplying (\ref{eq1b}) by $z$ and subtracting it from (\ref{eq2b}) we find that $$h^*(z)=\widehat a(z)-z a^*(z).$$ Then by Corollary \ref{reversing}, $h_P(z)$ is obtained by reversing the $h^*(z)$ numerator polynomial, 

\begin{equation}\label{eq3b}
h_P(z)=z^{d+1}h^*\left(\tfrac 1z\right)=z^{d+1}\widehat a\left(\tfrac 1z\right)-z^{d} a^*\left(\tfrac 1z\right)=\widehat a(z)- a^*(z). 
\end{equation}
If we take $a(z)=\widehat a(z)$ and $b(z)=a^*(z)$ we get the desired decomposition. Notice that $a(z)=z^{d+1}a\left(\frac{1}{z}\right)$ and $b(z)=z^{d}b\left(\frac{1}{z}\right)$, and both polynomials have nonnegative coefficients.
\end{proof}



\begin{proof}[Proof of Theorem \ref{th1}]
Given a poset $\Pi$ we are interested in the decomposition $h_{\Pi}(z) = a_{\Pi}(z) + zb_{\Pi}(z),$ with $a(z)=z^{d+1}a\left(\frac{1}{z}\right)$ and $b(z)=z^{d}b\left(\frac{1}{z}\right)$ of the numerator polynomial $h_{\Pi}(z)$ of the Ehrhart series of the open polytope $O^{\circ}_{\Pi}$. This decomposition is unique due to Proposition \ref{abdecomp}. By Corollary \ref{reversing}, $h_{\Pi}(z)$ is obtained by $h_{\Pi}(z)=z^{d+1}h^*\left(\frac{1}{z}\right),$ where $h^*(z)$ is the numerator polynomial of $\Ehr_{O_{\Pi}}(z)$.

For this we need to apply Theorem \ref{th3} to a polytope that is a projection of $O_\Pi$.
 Denote by $\mu$ the orthogonal projection from $\ensuremath{\mathbb{R}}^d$ to the $(d-1)$-dimensional space $H_0=\{x\in \ensuremath{\mathbb{R}}^d: \sum_{i=1}^d x_i=0\}$ and let $e_i$ be the canonical basis of $\ensuremath{\mathbb{R}}^d$ (for $1\le i\le d$).  The integer lattice $\ensuremath{\mathbb{Z}}^d$ is projected orthogonally to the lattice $L\subset H_0$ generated by $\mu(e_1), \ldots, \mu(e_{d-1})$. Note that $\mu(e_d)=-\sum_{i=1}^{d-1} \mu(e_i)$. Consider the projected polytope $\mu(O_{\Pi})\subseteq H_0$. It is a lattice polytope in $H_0$ so that it holds the following. 

\begin{lemma}
 The Ehrhart series of $O_\Pi$ and $\mu(O_\Pi)$ have the same numerator polynomial $h^*$, that is,  $$\Ehr_{\mu(O_\Pi)}(z)=\frac{h^*(z)}{(1-z)^d}.$$
 \end{lemma}

\begin{proof}
Let $T$ be the triangulation of $O_\Pi$ obtained by subdividing it through all hyperplanes of the form $x_i=x_j.$ 
Notice that  all unimodular simplices $\Delta$ in $T$ are projected by $\mu$ to unimodular simplices in $H_0$ (with respect to the lattice $L$). This is because the vertices of any of such simplex $\Delta$ contain a basis of $\ensuremath{\mathbb{Z}}^d$, and therefore their projections must generate $L$ by integer linear combinations. The points $(0, \ldots,0)$ and $(1, \ldots,1)$  are vertices of each simplex $\Delta$ in $T$ and both are projected to the origin in $H_0$. The projection of all other $d-1$ vertices of $\Delta$ is then an integer basis of $L$ and $\mu(\Delta)$ is unimodular. Therefore we obtain a unimodular triangulation $T_\mu$ for $\mu(O_\Pi)$ that comes from projecting all simplices $\Delta$ in $T$ by $\mu$.

If $f_\mu(z)$ is the $f$-polynomial of $T_\mu$, then the $f$-polynomial of $T$ is $$f_T(z)=f_{T_\mu}(z)(1+z),$$ since combinatorially $T$ is a cone over $T_\mu$ where the projection of the origin is split into two vertices in $T$ that belong to all maximal simplices. 
Using  (\ref{h}) we find that
\begin{align*}
h^*(z)=h_T(z) &= (1-z)^{d+1}f_T\left(\frac{z}{1-z}\right)=(1-z)^{d+1}\left(1+\frac{z}{1-z}\right)f_{T_\mu}\left(\frac{z}{1-z}\right)\nonumber\\              
              &=(1-z)^{d+1}\left(\frac{1}{1-z}\right)f_{T_\mu}\left(\frac{z}{1-z}\right)=(1-z)^{d}f_{T_\mu}\left(\frac{z}{1-z}\right)= h_{T_\mu} (z).
\end{align*}
The last equality is due to the fact that $\dim T_\mu=d-1,$ and this dimension is used in the corresponding formula for $h_{T_\mu}$ from  (\ref{h}).
\end{proof}
This result can be also seen by noticing that the triangulations $T$ and $T_\mu$ are always regular and therefore shellable, and they have the same $h$-vector. In \cite{HPPS} it is explained why order polytopes are compressed, and therefore have regular unimodular triangulations, and also how these properties are preserved under the projection $\mu$. Projected order polytopes are also examples of alcoved polytopes. See \cite{postnikovAP} for more information, including a description of the $h^*$-vector of $O_\Pi$ using the number of descents of the linear extensions of $\Pi$. 

Now we are ready to describe the decomposition of $h_{\Pi}(z).$
First apply Theorem \ref{stapledon} to polytope $O_{\Pi}$. There are symmetric polynomials $a^*(z)$ and $b^*(z)$ with nonnegative coefficients and with $a^*(z)=z^da^*\left(\frac{1}{z}\right)$ and $b^*(z)=z^{s-1}b^*\left(\frac{1}{z}\right)$ so that
\begin{equation}\label{eq1}
(1+\cdots+z^{l-1})h^*(z)=a^*(z)+z^l b^*(z). 
\end{equation}
Applying  Theorem \ref{stapledon} this time to polytope $\mu(O_\Pi)$ we  find polynomials $a_1^*(z)$ and $b_1^*(z)$ with nonnegative coefficients, and with  $a_1^*(z)=z^{d-1}a_1^*\left(\frac{1}{z}\right)$ and $b_1^*(z)=z^{s-1}b_1^*\left(\frac{1}{z}\right)$ so that
\begin{equation}\label{eq2}
(1+\cdots+z^{l-2})h^*(z)=a^*_1(z)+z^{l-1} b^*_1(z).
\end{equation}
In this case the value of $s$ is the same, but $d$ and $l$ have to be reduced by 1 to use the same notation as with $O_\Pi$. Also we can see that $b^*(z)=b^*_1(z)$, since the formula (\ref{abeq2}) for the coefficients of the $b$ polynomials does not depend on $d$ but only on $s$. 
\begin{lemma}
 If we decompose $h_{\Pi}(z)=a_\Pi(z)+zb_\Pi(z)$ as the sum of two symmetric polynomials $a_\Pi$ and $b_\Pi$, with $a_\Pi(z)=z^{d+1}a_\Pi\left(\frac{1}{z}\right)$ and $b_\Pi(z)=z^{d}b_\Pi\left(\frac{1}{z}\right)$, then necessarily $b_\Pi(z)= a^*(z)$ and $a_\Pi(z)=-za^*_1(z)$.
\end{lemma}
\begin{proof}
By Proposition \ref{abdecomp}, such a decomposition is unique (with $s=d+1$ playing the role of $d$ and $l=1$). So it is enough to check that the proposed polynomials $b_\Pi(z)= a^*(z)$ and $a_\Pi(z)=-za^*_1(z)$ satisfy the required conditions.  
Similar to the proof of Theorem \ref{th3}, if we multiply (\ref{eq2}) by $z$ and subtract it from (\ref{eq1}),  using that $b^*(z)=b^*_1(z)$  we can check that
\begin{equation*}
h^*(z)= a^*(z)-za^*_1(z).
\end{equation*}
From Corollary \ref{reversing} we have
\begin{equation*}
 h_{\Pi}(z)=z^{d+1}h^*\left(\tfrac{1}{z}\right)=z^{d+1}a^*\left(\tfrac 1z\right)-z^{d+1}\left(\tfrac 1z a^*_1\left(\tfrac 1z\right)\right)=za^*(z)-za^*_1(z).
\end{equation*}
If $b_\Pi(z)= a^*(z)$ and $a_\Pi(z)=-za^*_1(z)$,  we have that 
\begin{equation*}
 h_{\Pi}(z)=za^*(z)-za^*_1(z)=a_\Pi(z)+zb_\Pi(z).
\end{equation*}
 Also $b_\Pi(z)=z^{d}b_\Pi\left(\frac{1}{z}\right)$ and $z^{d+1}a_\Pi\left(\tfrac 1z\right)=-z^{d+1}\left(\tfrac 1z a^*_1\left(\tfrac 1z\right)\right)=-z^{d} a^*_1\left(\tfrac 1z\right)=a_\Pi(z),$
so these are symmetric polynomials that satisfy the desired conditions. 
\end{proof}

Now it is clear that the polynomial $b_\Pi$ has nonnegative coefficients while the polynomial $a_\Pi$ has nonpositive coefficients, due to the nonnegativity of the polynomials in the Stapledon decomposition. This completes the proof of Theorem \ref{th1}.
\end{proof}

\section{Inequalities for Chromatic Polynomials}

\begin{proof}[Proof of Theorem \ref{th2}]
Due to Proposition \ref{chromatic as ao} and since $\Lp_{O^{\circ}_\Pi}(n)=\Omega^{\circ}_\Pi(n-1),$ 
\begin{align*}
\sum_{n= 1}^{\infty}\chi_G(n-1)z^n &= \sum_{n= 1}^{\infty} \sum_{\rho\text{ acyclic}} \Omega^{\circ}_{\Pi_{\rho}}(n-1)z^n\\
                                     &=\sum_{\rho\text{ acyclic}} \sum_{n= 1}^{\infty} L^{\circ}_{\Pi_{\rho}}(n)z^n    =    \sum_{\rho\text{ acyclic}} \Ehr_{O^{\circ}_{\Pi_\rho}}(z) 
\end{align*}
 and then the numerator polynomial of this series is 
 \begin{equation}\label{hsum}
 zh_G(z)=\sum_{\rho\text{ acyclic}} h_{\Pi_{\rho}}(z), 
 \end{equation}
namely the sum of the corresponding numerator polynomials of the Ehrhart series of the open order polytopes corresponding to the posets $\Pi_\rho$ for all acyclic orientations $\rho$.

Now to find the decomposition $zh_G(z)=a(z)+zb(z)$,  with $a(z)$ and $b(z)$ symmetric so that $a(z)=z^{d+1}a\left(\frac{1}{z}\right)$ and $b(z)=z^{d}b\left(\frac{1}{z}\right)$, it is enough to find the corresponding decomposition for the polynomials $h_{\Pi_{\rho}}(z)$ and add them up together, since the formulas  (\ref{abeq1}) and (\ref{abeq2}) are linear, and all polynomials $h_{\Pi_{\rho}}(z)$ have degree $s=d+1$ and $l=1$.

So if $h_{\Pi_{\rho}}(z)=a_{\Pi_{\rho}}(z)+zb_{\Pi_{\rho}}(z)$ for polynomials as claimed in Theorem \ref{th1}, then the decomposition for $zh_G(z)$ is given by 
\begin{equation}\label{abhat}
a(z)=\sum_{\rho\text{ acyclic}} a_{\Pi_{\rho}}(z)\text{\qquad and \qquad}b(z)=\sum_{\rho\text{ acyclic}} b_{\Pi_{\rho}}(z), 
\end{equation}
 and it is clear that necessarily $b(z)$ and $-a(z)$ have nonnegative coefficients. 
 \end{proof}
 
\begin{proof}[Proof of Corollary \ref{cor}]
From Corollary \ref{reversing} and  (\ref{hsum}) we can see that $zh_G(z)$ is a polynomial of degree $d+1$ with nonnegative coefficients, and therefore $\deg(h_G)=d$. Moreover we observe that the leading coefficient $h_{d}$ is equal to the number of acyclic orientations of $G$. 
 Due to Proposition \ref{abdecomp} applied to $zh_G(z)$ and equations (\ref{abeq1}) and (\ref{abeq2}) we can find that the coefficients of the polynomials $a(z)$ and $b(z)$ are
 \begin{align}
a_i &=0+h_0+h_1+\cdots+h_{i-1}-h_{d}-\cdots-h_{d-i+1},\label{abh1}\\
b_i &=-0 -h_0-h_1-\cdots-h_{i-1}+h_{d}+\cdots+h_{d-i},\label{abh2}
\end{align}
where $a(z)=\sum_{i=0}^{d+1} a_iz^i$ and $b(z)=\sum_{i=0}^{d} b_iz^i$. Here the decomposition is made with respect to degree $s=d+1,$ and $l=1$. From Theorem \ref{th2} we get the inequalities  $-a_i\ge 0$  for all $i$ between $1$ and $\lfloor \frac{d+1}{2}\rfloor$, that give us the desired result. 
 \end{proof}
Notice that the inequalities $b_i\le 0$ for $0\le i\le d$ are weaker than those of the form  $-a_i\ge 0$, since $b_i=-a_i+h_{d-i}\ge -a_i.$
Since the polynomial $a(z)$ is symmetric, we do not get anything new for $i>\lfloor \frac{d+1}{2}\rfloor$.


\section{Further Conjectures}

We would like to know if the following conjecture holds.

\begin{conjecture}\label{c2}
Let $G$ be a graph with $d$ vertices. Then $h_G(z)$ can be decomposed as $h_G(z)=a(z)+zb(z)$ where $a(z)=z^{d}a\left(\frac{1}{z}\right)$ and $b(z)=z^{d-1}b\left(\frac{1}{z}\right)$ are symmetric polynomials such that $b(z)$ and $-a(z)$ have nonnegative coefficients.
\end{conjecture}

The following is the corresponding conjecture for open order polytopes.
\begin{conjecture}\label{c1}
Let $O_{\Pi}$ be an order polytope of dimension $d$ and  Let $p_{\Pi}$ be the numerator polynomial of the series $$\sum_{n= 1}^{\infty} \Omega^{\circ}_{\Pi}(n)z^n=\frac{p_{\Pi}(z)} {(1-z)^{d+1}}.$$ Then $p_{\Pi}$ can be decomposed as $p_{\Pi}(z)=a(z)+zb(z)$ where the polynomials $a(z)$ and $b(z)$ are symmetric with $a(z)=z^{d}a\left(\frac{1}{z}\right)$ and $b(z)=z^{d-1}b\left(\frac{1}{z}\right)$, and so that $b(z)$ and $-a(z)$ have nonnegative coefficients.
\end{conjecture}
Due to Theorem \ref{stanley} we have $\Lp_{O^{\circ}_\Pi}(n)=\Omega^{\circ}_\Pi(n-1)$. Then the numerator polynomials of the corresponding series satisfy  $p_\Pi(z)=\frac 1z h_\Pi(z),$ since we shift the series one position.
Also, by Corollary \ref{reversing}, we can write $p_\Pi(z)=z^dh^*\left(\frac{1}{z}\right),$ where $h^*(z)$ denotes the numerator polynomial of $\Ehr_{O_\Pi}(z)$. This is always a polynomial, since $\deg(h^*)\le d$. 

It is clear that Conjecture \ref{c2} is a consequence of Conjecture \ref{c1}, since it is obtained simply by adding the corresponding decomposition of each poset $\Pi_\rho$ for all acyclic orientations $\rho$ of $G$. 

%

After some computer experimentation (mainly using the mathematical software Sage \cite{sagemath}), we believe these conjectures are true. Also, we can obtain Conjecture \ref{c2} if the following holds.

\begin{conjecture}\label{q1}
 For any order $d$-polytope $O_\Pi$ different from the $d$-cube $C_d$, there is a lattice polytope $P_\Pi$ of dimension $d-2$ so that the numerator polynomial of its Ehrhart series $\Ehr_{P_\Pi}(z)$ is equal to $h^*(z)$, the numerator polynomial of $\Ehr_{O_\Pi}(z)$. 
\end{conjecture}

\begin{proposition}
 Conjecture \ref{q1}  implies Conjecture \ref{c1}.
\end{proposition}

\begin{proof}
First let us consider the case when 
$\Pi$ is the poset on $d$ elements where no pair of elements are comparable. For this $\Pi$ it holds that $O_\Pi=[0,1]^d$ is the unit $d$-cube $C_d$. It is well known that in this case we have that  $$h^*(z)=\sum_{k=0}^{d-1} A(d,k) z^k$$ where the coefficients of the polynomial are given by the Eulerian numbers $A(d,k)$ that count the number of permutations of $\{1,\ldots,d \}$ that have exactly $k$ descents (see for example \cite[Section 2.2]{beckrobins}). This is a symmetric polynomial satisfying $h^*(z)=z^{d-1} h^*\left(\frac{1}{z}\right)$, whose coefficients satisfy many interesting combinatorial relationships.
In this case $$p_\Pi(z)=z^{d} h^*\left(\tfrac 1z\right)=zh^*(z),$$ and the decomposition in Conjecture \ref{c1} for $p_\Pi$ is given by $a(z)=0$ and $b(z)=h^*(z)$, which satisfies the required conditions. 

Now we consider the case where there is at least one order relationship between a pair of elements in $\Pi$. Then the order polytope $O_{\Pi}$ is different from the cube $C_d$, since we are subdividing it with at least one hyperplane. So by Conjecture \ref{q1} we can assume the existence of a $(d-2)$-polytope $P_{\Pi}$ such that the numerator polynomial of $\Ehr_{P_\Pi}(z)$ is precisely $h^*(z)$. 

Applying Theorem \ref{th3} to this polytope, we find that the numerator polynomial $h_{P_\Pi}$ of the open Ehrhart series $\Ehr_{P_\Pi}^\circ(z)$ is
$$h_{P_\Pi}(z)=a_{P_\Pi}(z)-b_{P_\Pi}(z)$$ where $a_{P_\Pi}(z)=z^{d-1}a_{P_\Pi}\left(\frac{1}{z}\right)$ and $b_{P_\Pi}(z)=z^{d-2}b_{P_\Pi}\left(\frac{1}{z}\right)$ and both polynomials $a_{P_\Pi}(z)$ and $b_{P_\Pi}(z)$ have nonnegative coefficients. 

By Corollary \ref{reversing}, $h_{P_\Pi}(z)=z^{d-1}h^*(z)$ and then 
$$p_\Pi(z)=z^dh^*\left(\tfrac 1z\right)=zh_{P_\Pi}(z)=za_{P_\Pi}(z)-zb_{P_\Pi}(z).$$
We can check that the desired decomposition from Conjecture \ref{c1} is given by $a(z)=-zb_{P_\Pi}(z)$ and $b(z)=a_{P_\Pi}(z)$. This decomposition is unique due to Lemma \ref{abdecomp}. Clearly $b(z)$ and $-a(z)$ have nonnegative coefficients.
\end{proof}

In case Conjecture \ref{q1} is true, we would like to know if it is possible to find one such polytope  $P_\Pi$  having a unimodular triangulation that is combinatorially equivalent to the triangulation obtained by intersecting $\mu (O_\Pi)$ with the boundary of $\mu (C_d)$. This would be stronger than Conjecture \ref{q1}. 

Analog to Corollary \ref{cor}, our conjectures would imply the following inequality for the coefficients of chromatic polynomials.
\begin{conjecture}\label{conjcor} Let $G$ be a graph on $d$ vertices and $h_G(z)=\sum_{i=0}^{d} h_iz^i$. Then for $1\le i\le \lfloor \frac{d}{2}\rfloor$ it holds that 
$$h_{d}+\cdots+h_{d-i+1}-h_0-\cdots-h_{i}\ge 0.$$
\end{conjecture}

\begin{proposition}\label{}
Conjecture \ref{c2} implies Conjecture \ref{conjcor}. 

\end{proposition}
 
\begin{proof}
In case Conjecture \ref{c2} holds we have the decomposition $h_G(z)=a(z)+zb(z)$ where $a(z)=z^{d}a\left(\frac{1}{z}\right)$, $b(z)=z^{d-1}b\left(\frac{1}{z}\right)$ and so that $b(z)$ and $-a(z)$ have nonnegative coefficients. Due to Proposition \ref{abdecomp} and equations (\ref{abeq1}) and (\ref{abeq2}), the coefficients of the polynomials $a(z)$ and $b(z)$ are
 \begin{align*}
a_i &=h_0+h_1+\cdots+h_i-h_{d}-\cdots-h_{d-i+1},\\
b_i &=-h_0-h_1-\cdots-h_i+h_{d}+\cdots+h_{d-i}.
\end{align*}
Here $s=d$ and $l=1$. From the inequalities  $-a_i\ge 0$  for all $i$ between $1$ and $\lfloor \frac d2\rfloor$ we get the desired result. 
 \end{proof}
As before, the inequalities $b_i\ge0$ and $-a_i\ge 0$ for $i>\lfloor \frac d2\rfloor$ give us nothing new. The inequalities in Conjecture \ref{conjcor} are stronger than those in Corollary \ref{cor}, since $h_{i+1}\ge 0$.


\subsubsection*{Acknowledgements}

I want to thank Matthias Beck for suggesting this problem and to him, Tristram Bogart and Julián Pulido for many useful conversations and comments on the script. Also thanks to the organizers of ECCO 2016 (Escuela Colombiana de Combinatoria 2016, Medellín, Colombia)  and to Universidad de los Andes for their support.

\bibliography{mybibtex}{}
\bibliographystyle{alpha}

\end{document}